\pdfoutput=1

\documentclass[11pt]{article}
\usepackage{jhtitle,amsbsy,amsmath,amsthm,amssymb,times,graphicx}
\usepackage{probmacros}
\usepackage[sort]{natbib}

\usepackage{comment}
\specialcomment{details}{\begin{quotation}\small}{\end{quotation}}
\excludecomment{details}

\def\baro{\vskip  .2truecm\hfill \hrule height.5pt \vskip  .2truecm}
\def\barba{\vskip -.1truecm\hfill \hrule height.5pt \vskip .4truecm}

\newcommand{\pcite}[1]{\citeauthor{#1}'s \citeyearpar{#1}}

\newtheorem{proposition}{Proposition}
\newtheorem{corollary}{Corollary}

\newtheorem{theorem}{Theorem}
\newtheorem{remark}{Remark}

\newcommand{\X}{{\mathsf{X}}}

\begin{document}

\title{On the Data Augmentation Algorithm for Bayesian Multivariate
  Linear Regression with Non-Gaussian Errors} \author{Qian Qin and
  James P. Hobert \\ Department of Statistics \\ University of
  Florida} \date{December 2015}


\maketitle

\begin{abstract}
  Let $\pi$ denote the intractable posterior density that results when
  the likelihood from a multivariate linear regression model with
  errors from a scale mixture of normals is combined with the standard
  non-informative prior.  There is a simple data augmentation
  algorithm (based on latent data from the mixing density) that can be
  used to explore $\pi$.  \citet{hobe:jung:khar:2015} recently
  performed a convergence rate analysis of the Markov chain underlying
  this MCMC algorithm in the special case where the regression model
  is \textit{univariate}.  These authors provide simple sufficient
  conditions (on the mixing density) for geometric ergodicity of the
  Markov chain.  In this note, we extend \pcite{hobe:jung:khar:2015}
  result to the multivariate case.
\end{abstract}

\section{Introduction}
  \label{sec:intro}
  
  Let $Y_1,Y_2,\dots,Y_n$ be independent $d$-dimensional random
  vectors from the multivariate linear regression model
\begin{equation}
  \label{eq:mreg}
  Y_i = \beta^T x_i   + \Sigma^{\frac{1}{2}} \varepsilon_i \;,
\end{equation}
where $x_i$ is a $p \times 1$ vector of known covariates associated
with $Y_i$, $\beta$ is a $p \times d$ matrix of unknown regression
coefficients, $\Sigma^{\frac{1}{2}}$ is an unknown scale matrix, and
$\varepsilon_1,\dots,\varepsilon_n$ are iid errors from a scale
mixture of multivariate normals; that is, from a density of the form
\begin{equation*}
  f_H(\varepsilon) = \int_{0}^{\infty} \frac{u^{\frac{d}{2}}}{(2
    \pi)^{\frac{d}{2}}} \, \exp \Big \{ -\frac{u}{2}
  \varepsilon^T \varepsilon \Big \} h(u) \, du \;,
\end{equation*}
where $h$ is the density function of some positive random variable.
We shall refer to $h$ as a \textit{mixing density}.  For example, when
$h$ is the density of a $\mbox{Gamma}(\frac{\nu}{2}, \frac{\nu}{2})$
random variable, then $f_H$ becomes the multivariate Student's $t$
density with $\nu > 0$ degrees of freedom, which, aside from a
normalizing constant, is given by $\big[ 1 + {\nu}^{-1} \varepsilon^T
\varepsilon\big]^{-\frac{d + \nu}{2}}$.

Let $Y$ denote the $n \times d$ matrix whose $i$th row is $Y_i^T$, and
let $X$ stand for the $n \times p$ matrix whose $i$th row is $x_i^T$,
and, finally, let $\varepsilon$ represent the $n \times d$ matrix
whose $i$th row is $\varepsilon_i^T$.  Using this notation, we can
state the $n$ equations in \eqref{eq:mreg} more succinctly as follows
\[
Y = X \beta + \varepsilon \Sigma^{\frac{1}{2}} \;.
\]
Let $y$ and $y_i$ denote the observed values of $Y$ and $Y_i$,
respectively.

Consider a Bayesian analysis of the data from the regression model
\eqref{eq:mreg} using an improper prior on $(\beta,\Sigma)$ that takes
the form $\omega (\beta , \Sigma) \propto |\Sigma|^{-a} \, I_{{\cal
    S}_d}(\Sigma)$ where ${\cal S}_d \subset \mathbb{R}^{\frac{d(d +
    1)}{2}}$ denotes the space of $d \times d$ positive definite
matrices.  Taking $a=(d+1)/2$ yields the standard non-informative
prior for multivariate location scale problems.  The joint density of
the data from model \eqref{eq:mreg} is, of course, given by
\[
f (y | \beta, \Sigma) = \prod_{i=1}^n \Bigg[ \int_{0}^{\infty}
\frac{u^{\frac{d}{2}}}{(2\pi)^{\frac{d}{2}} |\Sigma|^{\frac{1}{2}}}
\exp \bigg\{ -\frac{u}{2} \Big(y_i - \beta^T x_i\Big)^T
\Sigma^{-1}\Big(y_i - \beta^T x_i\Big) \bigg\} h(u) \, du \Bigg] \;.
\]
Define
\[
m(y) = \int_{{\cal S}_d} \int_{\mathbb{R}^{p \times d}}
f(y|\beta,\Sigma) \, \omega(\beta,\Sigma) \, d\beta \, d\Sigma \;.
\]
The posterior distribution is proper precisely when $m(y) < \infty$,
and is given by
\[
\pi(\beta, \Sigma|y) = \frac{f(y|\beta,\Sigma) \,
  \omega(\beta,\Sigma)}{m(y)} \;.
\]
Let $\Lambda$ denote the $n \times (p+d)$ matrix $(X:y)$.  The
following conditions are necessary for propriety:
\begin{enumerate}
  \item[($N1$)] $\mbox{rank}(\Lambda) = p+d \;;$
  \item[($N2$)] $n > p + 2d - 2a \;.$
\end{enumerate}
We assume throughout that $(N1)$ and $(N2)$ hold.  

There is a well-known DA algorithm that can be used to explore the
intractable posterior $\pi(\beta, \Sigma|y)$ \citep[see,
e.g.,][]{liu:1996}.  In order to state this algorithm, we must
introduce some additional notation.  For $z = (z_1,\dots,z_n)$, let
$Q$ be an $n \times n$ diagonal matrix whose $i$th diagonal element is
$z_i^{-1}$.  Also, define $\Omega = (X^T Q^{-1} X)^{-1}$ and $\mu =
(X^T Q^{-1} X)^{-1} X^T Q^{-1} y$.  We shall assume throughout the
paper that the mixing density $h$ satisfies the following condition
\[
\int_1^\infty u^{\frac{d}{2}} \, h(u) \, du < \infty \;.
\]
We will refer to this as ``condition ${\cal M}$.''  Finally, define a
parametric family of univariate density functions indexed by $s \ge 0$
as follows
\begin{equation}
  \label{eq:psi}
  \psi(u;s) = c(s) u^{\frac{d}{2}} \, e^{ -\frac{s u}{2}} \, h(u) \;,
\end{equation}
where $c(s)$ is the normalizing constant.  When $h$ is a standard
density, $\psi$ often turns out to be one as well, but even in
non-standard cases, it's typically straightforward to make draws from
$\psi(\cdot,s)$ \citep{hobe:jung:khar:2015}.  The DA algorithm calls
for draws from the inverse Wishart ($\mbox{IW}_d$) and matrix normal
($\mbox{N}_{p,d}$) distributions.  The precise forms of the densities
are given in the Appendix.  We now present the DA algorithm.  If the
current state of the DA Markov chain is $(\beta_m,\Sigma_m) =
(\beta,\Sigma)$, then we simulate the new state,
$(\beta_{m+1},\Sigma_{m+1})$, using the following three-step
procedure.

\baro \vspace*{2mm}
\noindent {\rm Iteration $m+1$ of the DA algorithm:}
\begin{enumerate}
\item Draw $\{Z_i\}_{i=1}^n$ independently with $Z_i \sim \psi \Big(
  \cdot \;; \big( \beta^T x_i - y_i)^T \Sigma^{-1} \big( \beta^T x_i -
  y_i) \Big)$, and call the result $z = (z_1,\dots,z_n)$.
\item Draw
\[
\Sigma_{m+1} \sim \mbox{IW}_d \bigg( n-p+2a-d-1, \Big( y^T Q^{-1} y -
  \mu^T \Omega^{-1} \mu \Big)^{-1} \bigg)  \;.
\]
\item Draw $\beta_{m+1} \sim \mbox{N}_{p,d} \big( \mu, \Omega,
  \Sigma_{m+1} \big)$ \vspace*{-2.5mm}
\end{enumerate}
\barba

Denote the DA Markov chain by $\Phi =
\{(\beta_m,\Sigma_m)\}_{m=0}^\infty$, and its state space by $\X :=
\mathbb{R}^{p \times d} \times {\cal S}_d$.  For positive integer $m$,
let $k^m: \X \times \X \rightarrow (0,\infty)$ denote the $m$-step
Markov transition density (Mtd) of $\Phi$, so that if $A$ is a
measurable set in $\X$,
\[
P \Big((\beta_m,\Sigma_m) \in A \, \big| \, (\beta_0,\Sigma_0) =
(\beta,\Sigma) \Big) = \int_A k^m\big( (\beta',\Sigma') \big|
(\beta,\Sigma) \big) \, d\beta' \, d\Sigma' \;.
\]
(The precise form of $k^1$ is given in Section~\ref{sec:proof}.)  If
there exist $M: \X \rightarrow [0,\infty)$ and $\rho \in [0,1)$ such
that, for all $m$,
\begin{equation}
  \label{eq:ge}
  \int_{{\cal S}_d} \int_{\mathbb{R}^{p \times d}} \Big| k^m \big(
  \beta,\Sigma \big| \tilde{\beta},\tilde{\Sigma} \big) -
  \pi(\beta,\Sigma \big| y) \Big| \, d\beta \, d\Sigma \le
  M(\tilde{\beta},\tilde{\Sigma}) \, \rho^m \;,
\end{equation}
then the chain $\Phi$ is \textit{geometrically ergodic}.  (The
quantity on the left-hand side of \eqref{eq:ge} is the total variation
distance between the posterior distribution and the distribution of
$(\beta_m,\Sigma_m)$ conditional on $(\beta_0,\Sigma_0) =
(\tilde{\beta},\tilde{\Sigma})$.)  The main contribution of this paper
is to demonstrate that $\Phi$ is geometrically ergodic as long as $h$
converges to zero at the origin at an appropriate rate.  This result
is important from a practical perspective because geometric ergodicity
guarantees the existence of the central limit theorems that form the
basis of all the standard methods of calculating valid asymptotic
standard errors for MCMC-based estimators (see, e.g.,
\citet{robe:rose:1998}, \citet{jone:hobe:2001} and
\citet{fleg:hara:jone:2008}).

We now define three classes of mixing densities based on behavior near
the origin, and this will allow us to provide a formal statement of
our main result.  Let $h: \mathbb{R}_+ \rightarrow [0,\infty)$ be a
mixing density, where $\mathbb{R}_+ := (0,\infty)$.  If there is a
$\delta>0$ such that $h(u)=0$ for all $u \in (0,\delta)$, then we say
that $h$ is \textit{zero near the origin}.  Now assume that $h$ is
strictly positive in a neighborhood of 0 (i.e., $h$ is not zero near
the origin).  If there exists a $c>-1$ such that
\[
\lim_{u \rightarrow 0} \frac{h(u)}{u^c} \in (0,\infty) \;,
\]
then we say that $h$ is \textit{polynomial near the origin} with power
$c$.  Finally, if for every $c>0$, there exists an $\eta_c>0$ such
that the ratio $\frac{h(u)}{u^c}$ is strictly increasing in
$(0,\eta_c)$, then we say that $h$ is \textit{faster than polynomial
  near the origin}.  As shown in \citet{hobe:jung:khar:2015}
(henceforth, HJ\&K), all of the standard parametric families with
support $(0,\infty)$ are either polynomial near the origin, or faster
than polynomial near the origin.  Here is our main result.

\begin{theorem}
  \label{thm:main}
  Let $h$ be a mixing density that satisfies condition ${\cal M}$.
  Then the DA Markov chain is geometrically ergodic if $h$ is zero
  near the origin, or if $h$ is faster than polynomial near the
  origin, or if $h$ is polynomial near the origin with power $c >
  \frac{n-p+2a-d-1}{2}$.
\end{theorem}

\begin{remark}
  Theorem~\ref{thm:main} is the multivariate version of HJ\&K's
  univariate result.  However, because the parametrization used in
  HJ\&K is slightly different than that used here, setting $d=1$ in
  Theorem~\ref{thm:main} does not yield HJ\&K's result.  In
  particular, whereas we parametrize our model in terms of $\Sigma$,
  which is the natural parametrization in the multivariate setting,
  HJ\&K use $\sqrt{\Sigma}$.  Hence, we cannot directly compare the
  two results in the case $d=1$ since our hyperparameter $a$ has a
  different meaning than theirs.  To put the two models on the same
  footing when $d=1$, we would have to change our prior to $\omega^*
  (\beta, \Sigma) \propto |\Sigma|^{-\frac{a+1}{2}} \, I_{{\cal
      S}_d}(\Sigma)$.  Note that, if we set $d=1$ and replace $a$ by
  $(a+1)/2$ in Theorem~\ref{thm:main}, then we do indeed recover
  HJ\&K's result.
\end{remark}

\begin{remark}
  Fix $\nu>0$ and suppose that the mixing density is
  $\mbox{Gamma}(\frac{\nu}{2}, \frac{\nu}{2})$, which is clearly
  polynomial near the origin with power $\frac{\nu}{2}-1$.  It follows
  from Theorem~\ref{thm:main} that the DA Markov chain is
  geometrically ergodic as long as $\nu > n-p+2a-d+1$.  In particular,
  when $a = (d+1)/2$, we need $\nu > n-p+2$.  This special case of
  Theorem~\ref{thm:main} was established in \citet{roy:hobe:2010}.
\end{remark}

\section{Proof of the main result}
\label{sec:proof}

In order to formally define the Markov chain that the DA algorithm
simulates, we must introduce the latent data model.  Suppose that,
conditional on $(\beta,\Sigma)$, $\{(Y_i,Z_i)\}_{i=1}^n$ are iid pairs
such that
\[
Y_i| Z_i=z_i \sim \mbox{N}_d \big( \beta^T x_i, \Sigma
/ z_i \big)
\]
\[
Z_i \sim h \;.
\]
Let $z = (z_1,\dots,z_n)$ and denote the joint density of
$\{Y_i,Z_i\}_{i=1}^n$ by $\tilde{f}(y, z \big| \beta, \Sigma)$.  It's
easy to see that $\int_{\mathbb{R}_+^n} \tilde{f}(y, z \big| \beta,
\Sigma) \, dz = f(y \big| \beta, \Sigma)$.  Thus, if we define
\[
\pi(\beta, \Sigma,z \big| y) = \frac{\tilde{f}(y, z \big| \beta,
  \Sigma) \, \omega (\beta , \Sigma)}{m(y)} \;,
\]
then we have $\int_{\mathbb{R}_+^n} \pi(\beta, \Sigma,z \big| y) \, dz
= \pi(\beta, \Sigma \big| y)$ which is the posterior (target) density.
From here on, to simplify notation, we will write
$\pi(\beta,\Sigma,z)$ instead of $\pi(\beta,\Sigma,z|y)$.  The Mtd of
the chain underlying the DA algorithm is given by
\begin{equation*}
  k \big( \beta,\Sigma \big| \tilde{\beta},\tilde{\Sigma} \big) =
  \int_{\mathbb{R}^n_+} \pi(\beta,\Sigma|z) \,
  \pi(z|\tilde{\beta},\tilde{\Sigma}) \, dz \;,
\end{equation*}
where $\pi(\beta,\Sigma|z)$ and $\pi(z|\beta,\Sigma)$ are conditional
densities associated with $\pi(\beta,\Sigma,z)$.  The precise forms of
these densities can be gleaned from the statement of the DA algorithm
given in the Introduction.  (Note that the algorithm exploits the
representation $\pi(\beta, \Sigma \big| z) = \pi(\beta \big|
\Sigma,z) \pi(\Sigma|z)$.)  An argument used in Section 2 of
HJ\&K can be used here to show that $k$ is strictly positive on $\X
\times \X$, and this implies that (when the posterior is proper) the
DA Markov chain is Harris recurrent.  We now use a standard drift and
minorization argument to develop a sufficient condition for geometric
ergodicity of $\Phi$.

\begin{proposition}
  \label{prop:drift_smn}
  Suppose that there exist $\lambda \in \big[ 0, \frac{1}{n-p+2a-1}
  \big)$ and $L \in \mathbb{R}$ such that
\begin{equation}
   \label{eq:key}
   \frac{\int_0^\infty u^{\frac{d-2}{2}}  \, e^{ -\frac{s u}{2}} \, h(u)
     \, du}{\int_0^\infty u^{\frac{d}{2}} \, e^{ -\frac{s u}{2}} \, h(u) \, du}
   \le \lambda s + L
\end{equation}
for every $s \ge 0$.  Then the DA Markov chain is geometrically
ergodic.
\end{proposition}

\begin{proof}
  We will prove the result by establishing a drift condition and an
  associated minorization condition, as in \pcite{rose:1995} Theorem
  12.  Our drift function, $V: \mathbb{R}^{p \times d} \times {\cal
    S}_d$, is as follows
\begin{equation*}
  V(\beta,\Sigma) = \sum_{i=1}^n \big(y_i - \beta^T x_i)^T \Sigma^{-1}
  \big(y_i - \beta^T x_i) \;.
\end{equation*}

\noindent {\bf Part I: Minorization}.  Fix $l>0$ and define
\[
B_l = \big \{ (\beta,\Sigma) : V(\beta,\Sigma) \le l \big \} \;.
\]
We will construct $\epsilon \in (0,1)$ and a density function $f^*:
\mathbb{R}^{p \times d} \times {\cal S}_d \rightarrow [0,\infty)$
(both of which depend on $l$) such that, for all
$(\tilde{\beta},\tilde{\Sigma}) \in B_l$,
\[
k(\beta,\Sigma | \tilde{\beta},\tilde{\Sigma}) \ge \epsilon
f^*(\beta,\Sigma) \;.
\]
This is the minorization condition.  It suffices to construct
$\epsilon \in (0,1)$ and a density function $\hat{f}: \mathbb{R}_+^n
\rightarrow [0,\infty)$ such that, for all
$(\tilde{\beta},\tilde{\Sigma}) \in B_l$,
\[
\pi(z|\tilde{\beta},\tilde{\Sigma}) \ge \epsilon \hat{f}(z) \;.
\]
Indeed, if such an $\hat{f}$ exists, then for all
$(\tilde{\beta},\tilde{\Sigma}) \in B_l$, we have
\begin{equation*}
  k \big( \beta,\Sigma \big| \tilde{\beta},\tilde{\Sigma} \big) =
  \int_{\mathbb{R}^n_+} \pi(\beta,\Sigma|z) \,
  \pi(z|\tilde{\beta},\tilde{\Sigma}) \, dz \ge \epsilon
  \int_{\mathbb{R}^n_+} \pi(\beta,\Sigma|z) \, \hat{f}(z) \, dz
  = \epsilon f^*(\beta,\Sigma) \;.
\end{equation*}
We now build $\hat{f}$.  Define $\tilde{r}_i = \big(y_i -
\tilde{\beta}^T x_i)^T \tilde{\Sigma}^{-1} \big(y_i - \tilde{\beta}^T
x_i)$, and note that
\begin{equation*}
  \pi(z|\tilde{\beta},\tilde{\Sigma}) = \prod_{i=1}^n
  \psi(z_i;\tilde{r}_i) = \prod_{i=1}^n c(\tilde{r}_i) z^{\frac{d}{2}}_i \, e^{
    -\frac{\tilde{r}_i z_i}{2}} \, h(z_i) \;.
\end{equation*}
Now, for any $s \ge 0$, we have
\[
c(s) = \frac{1}{\int_0^\infty u^{\frac{d}{2}} \, e^{ -\frac{s u}{2}}
  \, h(u) \, du} \ge \frac{1}{\int_0^\infty u^{\frac{d}{2}} \, h(u) \,
  du} \;.
\]
By definition, if $(\tilde{\beta},\tilde{\Sigma}) \in B_l$, then
$\sum_{i=1}^n \tilde{r}_i \le l$, which implies that $\tilde{r}_i \le
l$ for each $i=1,\dots,n$.  Thus, if $(\tilde{\beta},\tilde{\Sigma})
\in B_l$, then for each $i=1,\dots,n$, we have
\[
z^{\frac{d}{2}}_i \, e^{ -\frac{\tilde{r}_i z_i}{2}} \, h(z_i) \ge
z^{\frac{d}{2}}_i \, e^{ -\frac{l z_i}{2}} \, h(z_i) \;.
\]
Therefore,
\begin{align*}
  \pi(z|\tilde{\beta},\tilde{\sigma}) & \ge \bigg[ \int_0^\infty
  u^{\frac{d}{2}} \, h(u) \, du \bigg]^{-n} \prod_{i=1}^n
  z^{\frac{d}{2}}_i \, e^{- \frac{l z_i}{2}} \, h(z_i) \\ & = \bigg[
  \frac{\int_0^\infty u^{\frac{d}{2}} \, e^{- \frac{l u}{2}} \, h(u)
    \, du}{\int_0^\infty u^{\frac{d}{2}} \, h(u) \, du} \bigg]^n
  \prod_{i=1}^n \frac{z^{\frac{d}{2}}_i \, e^{- \frac{l z_i}{2}} \,
    h(z_i)}{\int_0^\infty u^{\frac{d}{2}} \, e^{- \frac{l u}{2}} \,
    h(u) \, du} \\ & := \epsilon \hat{f}(z) \;.
\end{align*}
Hence, our minorization condition is established.

\noindent {\bf Part II: Drift}.  To establish the required drift
condition, we need to bound the expectation of
$V(\beta_{m+1},\Sigma_{m+1})$ given that $(\beta_m,\Sigma_m) =
(\tilde{\beta},\tilde{\Sigma})$.  This expectation is given by
\begin{align*}
  \int_{{\cal S}_d} \int_{\mathbb{R}^{p \times d}} & V(\beta,\Sigma)
  \, k(\beta,\Sigma | \tilde{\beta},\tilde{\Sigma}) \, d\beta \,
  d\Sigma
  \nonumber \\
  & = \int_{\mathbb{R}_+^n} \Bigg \{ \int_{{\cal S}_d} \bigg[
  \int_{\mathbb{R}^{p  \times d}} V(\beta,\Sigma) \, \pi(\beta|\Sigma,z) \,
  d\beta \bigg] \, \pi(\Sigma|z) \, d\Sigma \Bigg \}
  \pi(z|\tilde{\beta},\tilde{\Sigma}) \, dz \;.
\end{align*}
Calculations in \pcite{roy:hobe:2010} Section 4 show that
\begin{equation*}
  \int_{{\cal S}_d} \bigg[
  \int_{\mathbb{R}^{p  \times d}} V(\beta,\Sigma) \, \pi(\beta|\Sigma,z) \,
  d\beta \bigg] \, \pi(\Sigma|z) \, d\Sigma \le (n-p+2a-1) \sum_{i=1}^n
  \frac{1}{z_i} \;.
\end{equation*}
It follows from \eqref{eq:key} that
\begin{align*}
  \int_{\mathbb{R}_+^n} \Bigg \{ \int_{{\cal S}_d} \bigg[
  \int_{\mathbb{R}^{p \times d}} V(\beta,\Sigma) \,
  \pi(\beta|\Sigma,z) & \, d\beta \bigg] \, \pi(\Sigma|z) \, d\Sigma
  \Bigg \} \pi(z|\tilde{\beta},\tilde{\Sigma}) \, dz \\ & \le
  (n-p+2a-1) \int_{\mathbb{R}_+^n} \bigg[ \sum_{i=1}^n \frac{1}{z_i}
  \bigg] \pi(z|\tilde{\beta},\tilde{\sigma}) \, dz \\ & = (n-p+2a-1)
  \sum_{i=1}^n c(\tilde{r}_i) \int_0^\infty
  u^{\frac{d-2}{2}} \, e^{ -\frac{\tilde{r}_i u}{2}} \, h(u) \, du \\
  & \le (n-p+2a-1) \bigg(\lambda \sum_{i=1}^n \tilde{r}_i + nL \bigg)
  \\ & = \lambda (n-p+2a-1) V(\tilde{\beta},\tilde{\Sigma}) +
  (n-p+2a-1) n L \\ & = \lambda' V(\tilde{\beta},\tilde{\sigma}) + L'
  \;,
\end{align*}
where $\lambda' := \lambda (n-p+2a-1) \in [0,1)$ and $L' := (n-p+2a-1)
n L$.  Since the minorization condition holds for any $l>0$, we can
appeal to \pcite{rose:1995} Theorem 12 to get the result.  This
completes the proof.
\end{proof}

Suppose that $g$ is a mixing density.  For a positive integer $d$, we
say that $g$ satisfies Condition $\mathcal{A}_d$ with $\lambda \in
[0,\infty)$ if there exists $k_\lambda \in \mathbb{R}$ such that
\begin{equation*} 
  \frac{\int_0^\infty u^{\frac{d-2}{2}} \, e^{ -\frac{s u}{2}} \, g(u)
    \, du}{\int_0^\infty u^{\frac{d}{2}} \, e^{ -\frac{s u}{2}} \, g(u) \, du}
  \le \lambda s + k_\lambda
\end{equation*}
for every $s \ge 0$.  The following result can be proven using the
results in HJ\&K's Section 4.

\begin{theorem}
  \label{thm:lambda}
  Suppose that $g$ is a mixing density such that $\int_1^\infty
  \sqrt{u} \, g(u) \, du < \infty$.  If $g$ is either zero near the
  origin or faster than polynomial near the origin, then $g$ satisfies
  Condition $\mathcal{A}_1$ with any $\lambda > 0$.  If $g$ is
  polynomial near the origin with power $c>-\frac{1}{2}$, then $g$
  satisfies Condition $\mathcal{A}_1$ with any $\lambda >
  \frac{1}{2c+1}$.
\end{theorem}

We use this result to prove the following:

\begin{corollary}
  \label{cor:lambda}
  Fix a positive integer $d$ and suppose that $g$ is a mixing density
  with $\int_1^\infty u^{\frac{d}{2}} \, g(u) \, du < \infty$.  If $g$
  is either zero near the origin or faster than polynomial near the
  origin, then $g$ satisfies Condition $\mathcal{A}_d$ with any
  $\lambda > 0$.  If $g$ is polynomial near the origin with power
  $c>-\frac{1}{2}$, then $g$ satisfies Condition $\mathcal{A}_d$ with
  any $\lambda > \frac{1}{2c+d}$.
\end{corollary}

\begin{proof}
  Let $g^*(u)$ be the mixing density that is proportional to
  $u^{\frac{d-1}{2}} \, g(u)$.  Then $\int_1^\infty \sqrt{u} \, g^*(u)
  \, du < \infty$, so Theorem~\ref{thm:lambda} applies to $g^*$, and
  \begin{equation}
    \label{eq:star}
    \frac{\int_0^\infty u^{-\frac{1}{2}} \, e^{ -\frac{s u}{2}} \,
      g^*(u) \, du}{\int_0^\infty u^{\frac{1}{2}} \, e^{ -\frac{s u}{2}}
      \, g^*(u) \, du} = \frac{\int_0^\infty u^{\frac{d-2}{2}} \, e^{
        -\frac{s u}{2}} \, g(u) \, du}{\int_0^\infty u^{\frac{d}{2}} \,
      e^{ -\frac{s u}{2}} \, g(u) \, du} \;.
  \end{equation}
  Now, it's easy to see that, if $g$ is zero near the origin, then so
  is $g^*$, and if $g$ is faster than polynomial near the origin, then
  so is $g^*$.  Hence, if $g$ is either zero near the origin or faster
  than polynomial near the origin, then by Theorem~\ref{thm:lambda},
  $g^*$ satisfies Condition $\mathcal{A}_1$ with any $\lambda > 0$,
  which implies (by \eqref{eq:star}) that $g$ satisfies Condition
  $\mathcal{A}_d$ with any $\lambda > 0$.  Finally, assume that $g$ is
  polynomial near the origin with power $c>-\frac{1}{2}$.  Then $g^*$
  is polynomial near the origin with power $(2c+d-1)/2 >
  -\frac{1}{2}$.  Theorem~\ref{thm:lambda} implies that $g^*$
  satisfies Condition $\mathcal{A}_1$ with any $\lambda > 1/(2c+d)$.
  It follows from \eqref{eq:star} that $g$ satisfies Condition
  $\mathcal{A}_d$ with any $\lambda > 1/(2c+d)$.
\end{proof}

\begin{proof}[Proof of Theorem~\ref{thm:main}]
  Assume that $h$ is zero near the origin or faster than polynomial
  near the origin.  Corollary~\ref{cor:lambda} implies that $h$
  satisfies Condition $\mathcal{A}_d$ with any $\lambda > 0$, and it
  follows from Proposition~\ref{prop:drift_smn} that the DA Markov
  chain is geometrically ergodic.  Now assume that $h$ is polynomial
  near the origin with power $c > \frac{n-p+2a-d-1}{2}$, which is
  strictly positive by ($N2$).  Corollary~\ref{cor:lambda} implies
  that $h$ satisfies Condition $\mathcal{A}_d$ with any $\lambda >
  1/(2c+d)$, and since
\[
\frac{1}{2c+d} < \frac{1}{n-p+2a-1} \;,
\]
Proposition~\ref{prop:drift_smn} implies that the DA Markov chain is
geometrically ergodic.
\end{proof}

\vspace*{10mm}

\noindent {\LARGE \bf Appendix}
\begin{appendix}

\begin{description}
\item[Matrix Normal Distribution] Suppose $Z$ is an $r \times c$
  random matrix with density
\[
f_{Z}(z) = \frac{1}{(2\pi)^{\frac{rc}{2}} |A|^{\frac{c}{2}}
  |B|^{\frac{r}{2}}} \exp \bigg[ -\frac{1}{2}\mbox{tr} \Big\{ A^{-1}(z
  - \theta) B^{-1} (z - \theta)^T \Big\} \bigg] \;,
\]
where $\theta$ is an $r \times c$ matrix, $A$ and $B$ are $r \times r$
and $c \times c$ positive definite matrices.  Then $Z$ is said to have
a \textit{matrix normal distribution} and we denote this by $Z \sim
\mbox{N}_{r,c} (\theta,A,B)$ \citep[][Chapter 17]{arno:1981}.

\item[Inverse Wishart Distribution] Suppose $W$ is a $d \times d$
  random positive definite matrix with density
\[
f_{W}(w) = \frac{|w|^{-\frac{m+d+1}{2}} \exp \Big \{ -\frac{1}{2}
  \mbox{tr} \big( \Theta^{-1} w^{-1} \big) \Big\}}{ 2^{\frac{md}{2}}
  \pi^{\frac{d(d-1)}{4}} |\Theta|^{\frac{m}{2}} \prod_{i=1}^d \Gamma
  \big( \frac{1}{2}(m+1-i) \big)} I_{{\cal S}_d}(W) \;,
\]
where $m > d-1$ and $\Theta$ is a $d \times d$ positive definite
matrix.  Then $W$ is said to have an \textit{inverse Wishart
  distribution} and this is denoted by $W \sim \mbox{IW}_d(m,
\Theta)$.
\end{description}

\end{appendix}

\vspace*{5mm}

\noindent {\bf \large Acknowledgment}.  The second author was
supported by NSF Grant DMS-15-11945.

\bibliographystyle{ims}
\bibliography{refs}

\end{document}